\documentclass[amsthm]{elsart}

\usepackage{yjsco}
\usepackage{natbib}
\usepackage{amsthm}
\usepackage{enumerate}
\usepackage{amsmath}
\usepackage{amssymb}
\usepackage{amsfonts}
\usepackage{epsfig}
\usepackage[matrix,arrow,curve]{xy}
\usepackage{color,graphicx}

\usepackage[breaklinks,bookmarksopen,bookmarksnumbered]{hyperref}
\hypersetup{colorlinks=true,backref=true,citecolor=blue,}
\newcommand{\Mac}{{\textit {Macaulay2}}}

\newcommand{\tensor}{\otimes}

\newcommand{\sO}{\mathcal{O}}

\newcommand{\cp}{\mathbb{C}}
\newcommand{\qp}{\mathbb{Q}}
\newcommand{\PP}{\mathbb{P}}
\newcommand{\pp}{\mathbb{P}}
\newcommand{\np}{\mathbb{N}}

\newcommand{\zp}{\mathbb{Z}}

\newcommand{\wt}{\widetilde}

\DeclareMathOperator{\chara}{char}
\DeclareMathOperator{\Cliff}{Cliff}
\DeclareMathOperator{\gon}{gon}

\DeclareMathOperator{\Pic}{Pic}

\DeclareMathOperator{\rank}{rank}

\DeclareMathOperator{\T}{T}

\DeclareMathOperator{\h}{H}
\DeclareMathOperator{\Hom}{Hom}
\DeclareMathOperator{\Ann}{Ann}
\DeclareMathOperator{\coker}{coker}
\DeclareMathOperator{\Tor}{Tor}
\DeclareMathOperator{\codim}{codim}
\DeclareMathOperator{\Sing}{Sing}

\DeclareMathOperator{\Sym}{Sym}

\newcommand{\sHom}{\mathcal{H}om}

\theoremstyle{plain}
\newtheorem{theorem}{Theorem}[section]
\newtheorem{lemma}[theorem]{Lemma}
\newtheorem{proposition}[theorem]{Proposition}

\newtheorem{conjecture}[theorem]{Conjecture}

\theoremstyle{definition}
\newtheorem{remark}[theorem]{Remark}
\newtheorem{definition}[theorem]{Definition}
\newtheorem{example}[theorem]{Example}
\newtheorem{algo}[theorem]{Algorithm}

\begin{document}
\begin{frontmatter}

\title{Computational aspects of gonal maps and radical parametrization of curves}

\author[Aut1]{J. Schicho} 
\ead{Josef.Schicho@oeaw.ac.at},
\author[Aut2]{F.-O. Schreyer}
\ead{schreyer@math.uni-sb.de} \and
\author[Aut3]{M. Weimann} 
\ead{weimann@unicaen.fr}

\address[Aut1]{RICAM, Austrian Academy of Sciences, Altenbergerstrasse 69
A-4040 Linz, Austria}

\address[Aut2]{Mathematik und Informatik, Geb\"{a}ude E 2.4
Universit\"{a}t des Saarlandes
D-66123 Saarbr\"{u}cken}

\address[Aut3]{LMNO, Universit\'e de Caen BP 5186, F 14032 Caen Cedex}


\begin{abstract}
We develop in this 	article an algorithm that, given a projective curve $C$, computes a \textit{gonal map}, that is, a finite morphism from $C$ to $\pp^1$ of minimal degree. Our method is based on the computation of scrollar syzygies of canonical curves. We develop an improved version of our algorithm for curves with a unique gonal map and we discuss a characterization of such curves in terms of Betti numbers. Finally, we derive an efficient algorithm for radical parametrization of curves of gonality $\le 4$.
\end{abstract}


\end{frontmatter}

\section{Introduction}
The \textit{gonality} $\gon(C)$ of a projective curve $C$ is the smallest integer $d$ for which there exists a $d:1$ morphism to $\pp^1$.  A \textit{gonal map} is such a morphism of smallest degree. In this article, we adress the problem of computing a gonal map of a given curve $C$. This is an important problem of both computational interest (e.g. encoding curves by small function field extensions, or for radical parametrization) and of theoretical interest. Up to our knowledge, there do not exist such efficient algorithms, except for curves of gonality $\le 3$ \cite{SS} or genus $g\le 6$ \cite{har}.

We work over an arbitrary field, with the main focus on the case of an algebraically closed field of characteristic 0. In case of an non-algebraically closed field $K$, we want to compute the gonality over the algebraically closure and determine in addition the field of definition of the gonal  map. In more algebraic terms we want to find a algebraic field extension $L \supset K$ of the ground field such that the function field $L(C)$ is an algebraic extension of $L(t)=L(\pp^1)$ of degree $d=\gon(C)$. We solve this problem:

\begin{theorem} \label{t1} There exists a deterministic algorithm which, given $C$ an absolutely irreducible projective curve, computes the gonality of $C$ and a gonal map.
\end{theorem}

The algorithm we propose here is based on the computation of scrollar syzygies of the canonical model of $C$, a concept introduced by von Bothmer \cite{both}. However with current computational power, this algorithm rarely works in practice. 

\paragraph*{Goneric curves.} One might expect that there is a much more effective algorithm in case the gonal map is unique up to an automorphisms of $\pp^1$. Indeed we will describe such an algorithm for curves which satisfy an even stronger condition we call \textit{goneric}. 

\begin{theorem}\label{t2}
There exists a deterministic algorithm that, given a projective curve, outputs ``goneric'' and computes the gonal map or outputs ``non goneric''.
\end{theorem}

The precise definition of goneric curves involves graded Betti numbers and will be given later in Section \ref{S4}. Although a general curve of genus $g$ is not goneric, we expect general $d$-gonal curves of non maximal gonality to be goneric. We will adress a precise conjectural characterization of goneric curves in terms of special linear series. 

\paragraph*{Radical parametrization.} Detecting curves of gonality $\le 4$, and computing a gonal map, is of particular interest, since by Cardano's formula  these curves can be parametrized with radicals, \textit{i.e} by using $\sqrt{}$, $\sqrt[3]{}$. We prove here

\begin{theorem} \label{t3} There is a deterministic algorithm which decides whether a curve $C$ has gonality $\gon(C) \le 4$ and computes a rational radical parametization for these curves in case $\chara(K) \not=2,3$.
\end{theorem}

Such an algorithm is well-known for rational and hyperelliptic curves. It is solved for $3$-gonal curves in \cite{SS} for $\chara(K)=0$ and for curves of genus $\le 6$ in \cite{har}, by using similar ideas as those developed here. Although Theorem \ref{t3} is a direct consequence of Theorem \ref{t1}, the algorithm we provide here is much more efficient.

By a dimension count Zariski proved that for a general curve $C$ of genus $g>6$, the Galois group of the hull of $\cp(C) \supset \cp(f)$ is not solvable for every $f \in \cp(C) \setminus \cp$, i.e.  $C$ has no \textit{rational radical parametrization}. This does not exclude the possibility that there is a \textit{algebraic radical parametization}, i.e. a surjection $C' \to C$ from a curve $C'$ with a rational radical parametrization. These two concepts are really different, as shown in \cite{PS}. Zariski suggested that the general curve $C$ of a fixed large enough genus should not be algebraically parametrized by radicals. Up to our knowledge, there exists no example of a curve with no algebraic radical parametization.

\paragraph*{Idea of proof.} Both Theorems \ref{t1}, \ref{t2} and \ref{t3} are based on the study of the syzygies of canonical curves. Since gonality is a birational invariant and since there are efficient algorithms to compute the normalization map, there is no less to suppose that $C$ is a smooth and absolutely irreducible projective curve. Then we consider the canonical map 
$$
C \hookrightarrow  \pp^{g-1}=\pp(\Gamma(C,\omega_C))
$$
determined by the canonical sheaf of $C$ (it's an embedding if and only if the gonality $\ge 3$). 
Let $S =K[x_0, \ldots, x_{g-1}]$ denotes the homogeneous coordinate ring of $\PP^{g-1}$, and let $S_C=S/I_C$ denotes the homogeneous coordinate ring of $C$.
There are effective algorithms to compute $S_C$, say by computing generators of the homogeneous ideal $I_C$ of $C$.

A $d$-gonal canonical curve is easily seen to be contained in a $(d-1)$-dimensional rational normal scroll $X$, whose projection map to $\pp^1$ restricted to $C$ induces a $d$-gonal map $f:C\to \pp^1$. So there is a commutative diagram
\begin{equation}\label{eval}
\begin{array}{ccccc}
 C & & \subset & & X \\
 {f} & \searrow & \quad &  \swarrow  & \\
 &  & \pp^1 & & 
\end{array}
\end{equation}
(the right hand arrow might not be a morphism if $X$ is a cone). The minimal free resolution of the coordinate homogeneous ring $S_X$ is given by an Eagon Northcott complex:
Consider $f^* \sO_{\PP^1}(1) \cong \sO_C(D)$. Then $h^0(\sO_C(D))=2$ because in case of $h^0(\sO_C(D)) \ge 3$ we could pass to the pencil
$|D -p |$ and obtain  a map of smaller gonality. By Riemann-Roch gives $h^0(\omega_C(-D))=g-d+1$.
Thus the multiplication tensor 
$$
\h^0 (\sO_C(D)) \tensor \h^0 (\omega_C(-D)) \to \h^0 (\omega_C) \cong \h^0( \PP^{g-1}, \sO(1))
$$
defines a $2 \times (g-d+1)$ matrix $\varphi$ of linear forms, whose $2\times 2$ minors vanish on $C$. 
The homogeneous ideal $I_X$ of the rational normal scroll is defined by these minors, and the minimal free resolution $E_X$ is the Eagon-Northcott complex associated to $\varphi$. 

The central idea is to use syzygies to extract the homogeneous ideal of $X$ from that of $C$. 
In general, this can be done by looking for a linear syzygy of $C$ of minimal rank thanks to a theorem of von Bothmer \cite{both}, in which case we get directly the matrix $\varphi$ and the projection $X\to \pp^1$. This approach has a prohibitive cost in general. \textit{When the curve is goneric}, we can compute $I_X$ directly as the annihilator of the cokernel of some explicit submatrix of syzygies $I_C$. In such a case, we compute the projection $X\to \pp^1$ thanks to a recognition method based on the Lie Algebra of $Aut(X)$ developed for the parametrization of some rational surfaces in \cite{schi}. The gonal map then follows explicitely. When the gonality is $\le 4$, we can avoid Bothmer's approach to treat the exceptional non goneric cases thanks to a theorem of Schreyer \cite{S2} that allows to replace the scroll and its ruling in the diagramm below with either a Del Pezzo surface or an elliptic cone with a $\pp^1$-fibration.

\paragraph*{Organisation.} In Section \ref{S2}, we study projective graded resolution of canonical curves. We review well known facts about the relations beetween Betti numbers and special linear series and we recall the main results of Brill-Noether theory. In Section \ref{S3}, we explain the relations between scrolls, gonal maps and syzygies and we prove theorem \ref{t1}. In Section \ref{S4}, we introduce goneric curves. We first prove Theorem \ref{t2}. Then we propose a conjecture about a  characterization of goneric curves in terms of special linear series. In Section \ref{S5}, we focus on radical parametrization of tetragonal curves and we prove Theorem \ref{t3}. Our results and algorithms are illustrated by various examples, with related links to various \textit{Macaulay2} or \textit{Magma} packages available on the webpages of the authors.

This research was partially supported by the Austrian Science Fund (FWF): P22766-N18.

\section{Betti numbers of canonical curves and special linear systems}\label{S2}

In this section, we review the basic facts about the relations beetween projective resolution of canonical curves and special linear series. 

\subsection{Canonical embedding}
Let $C$ be a smooth irreducible curve with genus $g\ge 2$ and let $\omega_C$ be its canonical line bundle. The \textit{canonical map} 
$$
C \rightarrow  \pp^{g-1}=\pp(\Gamma(C,\omega_C))
$$
determined by $\omega_C$ is an embedding if and only if $C$ is non hyperelliptic. There are efficient algorithms to compute the homogeneous ideal of a smooth canonical model of any irreducible non hyperelliptic curve and we will assume from now that $C$ is a \textit{canonical curve} $C\subset \pp^{g-1}$ given by its homogeneous ideal 
$$
I_C\subset S:=k[x_0,\ldots,x_{g-1}]
$$ 
in the homogeneous coordinate ring $S$ of $\pp^{g-1}$. By Max Noether theorem \cite{Noet}, the homogeneous coordinate ring $S_C:=S/I_C$ of $C$ is projectively normal
$$
S_C\simeq \bigoplus_{n\ge 0} \h^0 \mathcal{O}_C(n)
$$
(isomorphic to the canonical ring) and $S_C$ is a Cohen-Macaulay $S$-module.

\subsection{Projective resolution and Betti numbers}\label{ss2.1}

The \textit{minimal resolution} of $S_C$  
\begin{equation}\label{resolution}
(F_{\bullet}):\quad 0\leftarrow S_C \leftarrow F_0 \leftarrow \cdots \leftarrow F_i \stackrel{f_i}\leftarrow F_{i+1} \leftarrow \cdots \leftarrow F_{g-2} \leftarrow 0 
\end{equation}
is an \textit{exact complexe} of graded free $S$-modules (the syzygies modules)
$$
F_i=\bigoplus_{j\in \np} S(-i-j)^{\beta_{i,i+j}},
$$
where $S(-i)$ denotes the module generated by polynomials of degree $i$. Since $C$ is Cohen-Macaulay, the length of the resolution is $g-2$ thanks to the Auslander-Buchsbaum formula \cite[Theorem A2.15]{Eis}. The exponents $\beta_{i,i+j}$ are called the \textit{graded Betti numbers} of $C$. We have equality
$$
\beta_{i,j}=\dim \Tor_i^S(S_C,K)_j
$$
\cite[Proposition 1.7]{Eis}. In particular, Betti numbers coincide with the dimension of graded pieces of the homology of the Koszul resolution of $K$ tensored by $S_C$. The Betti numbers obey to the symmetric property  
$$
\beta_{ij}=\beta_{g-2-j,g-2-i}
$$
inherent to the self-duality property 
$$
F_{\bullet}\simeq \Hom_S(F_{g-2-\bullet}, S(-g-1))
$$
of minimal resolution of canonical curves \cite[Proposition 9.5]{Eis}. Since the resolution is minimal, there are no degree $0$ syzygies and $\beta_{0,k}= 0$ except for $\beta_{0,0}=1$. By symmetry, this implies that $\beta_{j,k}=0$ for all $j\ge 3$ and all $k$, except $\beta_{3,g-2}=\beta_{0,0}=1$. 
In particular, canonical curves are generated by quadrics and cubics for $g\ge 3$ and admit only linear and quadratic syzygies
$$
F_i=S(-i-1)^{\beta_{i,i+1}}\oplus S(-i-2)^{\beta_{i,i+2}},
$$
except for the last syzygy module $F_{g-2}=S(-g-1)$. This is part of Petri's theorem, which asserts moreover that $C$ is generated only by quadrics ($\beta_{1,3}= 0$), except if $C$ either hyperelliptic, trigonal, or a smooth planar quintic canonically embedded.

We can collect all the Betti numbers in an array called the \textit{Betti table} of $C$. By the previous proposition, the Betti table of canonical curves has the following shape 
\begin{equation*}
\begin{array}{cccccccc}
1\quad   &  -  &  -  &     &    -   & -  & \quad -   \\ 
-\quad  & \quad \beta_{12} \quad & \quad \beta_{23} \quad & \quad \cdots \quad &  \beta_{g-4,g-3} & \beta_{g-3,g-2} & \quad -    \\ 
-\quad  &  \beta_{13}  &  \beta_{24}  & \quad \cdots \quad   &  \quad  \beta_{g-4,g-2} \quad &  \beta_{g-3,g-1}  & \quad -  \\
-\quad  &  -  & -  &     &   -     &   - &  \quad 1 \\
\end{array}
\end{equation*}
where $\beta_{i,i+1}=\beta_{g-i-2,g-i}$ and where $-$ stands for vanishing Betti numbers.

We can compute easily some of the Betti numbers: from the Hilbert function of canonical curves \cite[Corollary 9.4]{Eis} and its relation with Betti numbers \cite[Corollary 1.10]{Eis}, we get some explicit formulas for the difference
$$
\beta_{i,i+1}-\beta_{i-1,i+1}=\binom{g-2}{i-2}-(g-2)\binom{g-2}{i-1}+(g-2)\binom{g-2}{i}-\binom{g-2}{i+1}
$$
in terms of $i$ and $g$. In particular, the Betti numbers $\beta_{i,i+1}$ can be computed explicitely in terms on $g$ and $i$ for all $i\ge 1$ such that $\beta_{i-1,i+1}=0$. For instance, the number $\beta_{1,2}$ of linearly independant quadrics vanishing on $C$ is always equal to   
$$
\beta_{1,2}=\binom{g-2}{2}.
$$
(recall that $g\ge 3$). It's thus natural to introduce the following invariant:
\begin{definition}
The \textit{linear colenght} $\ell(C)$ of $C$ is the smallest integer $i$ such that the Betti number $\beta_{g-2-i,g-1-i}=\beta_{i,i+2}$ is non zero.
\end{definition}
For a fixed $g$, the linear colength $\ell$ can take different values. For instance, a non hyperelliptic curve of genus $6$ can have Betti tables
\[ \begin{array}{ccccccccccc}
  1 &  - &  - &  - &  -   &\qquad \qquad\qquad\qquad&	   1 &  - &  - &  - &  - \\
  - & 6 & 5 & - &  -  &\qquad\qquad {\rm or} \qquad\qquad &  - & 6 & 8 & 3 &  - \\
  - &  - &  5 &  6 & - 	  &\qquad \qquad\qquad\qquad&	   - &  3 & 8 &  6 & - \\
  - &  - &  - &  - &  1    &\qquad \qquad\qquad	\qquad &   - &  - &  - &  - &  1 \\
\end{array}  \]
For fixed $g$ and $\ell$, the Betti numbers $\beta_{i,i+1}$ for $i> \ell$ can take different values too. First examples are curves with $g=7$ and $\ell=2$ which can have Betti tables 
\[ \begin{array}{ccccccccccccc}
  1 &  - &  - &  - &  - & -  	& \qquad \qquad\qquad\qquad	&	   
  1 &  - &  - &  - &  - & - \\
  - & 10 & 16 & 3 &  -  & - 		& \qquad\qquad {\rm or} \qquad\qquad &
  - & 10 & 16 & 9  &  -  & - \\
  - &  - &  3 &  16 & 10 & -	  	&  \qquad \qquad\qquad\qquad   &	 
  - & - &  9 & 16 & 10   & - \\
  - &  - & - & - &  -  & 1  &	\qquad \qquad\qquad\qquad &   
  - &  - &  - &  - &  -  & 1 \\
\end{array}  \]
so that $\beta_{\ell+1,\ell+2}=3$ or $9$. See for instance \cite{S1} for all possible Betti tables for curves of genus $\le 8$, \cite{Sag} for genus $g=9$. In general, the shape of the Betti tables is related to the intrinsec geometry of the curve, especially to the existence of special linear series of small degrees, as the gonal pencil we're interested in.

\subsection{Special linear series.}\label{ss2.2} 

An effective divisor $D$ is \textit{special} if $h^1(D):=h^1(\mathcal{O}_C(D))>0$. Clifford's theorem asserts that the dimension of complete linear systems attached to a special divisor can not be too big when compared to the degree. Namely, any effective special divisor satisfies the inequality
$$
g+1-(h^0(D)+h^1(D))=\deg(D)- 2(h^0(D)-1)\ge 0
$$
(first equality is Riemann-Roch) and equality holds if and only if $D$ is the zero divisor, the canonical divisor or a multiple of a pencil of degree $2$ on a hyperelliptic curve. Hence the natural well-know definition:

\begin{definition}
The \textit{Clifford index} of a special divisor $D$ on $C$ is defined as
$$\Cliff(D):=\deg(D)- 2(h^0(D)-1).$$ 
The \textit{Clifford index} $\Cliff(C)$ of $C$ is the smallest Clifford index of a divisor $D$ of $C$ which satisfies $h^i(D)\ge 2$ for $i=0,1$, with the convention $\Cliff(C)=0$ for curves of genus $g\le 2$ .
\end{definition}

By Serre duality, we have that $\Cliff(D)=\Cliff(K-D)$ where $K$ is a canonical divisor on $C$. The conditions $h^i(D)\ge 2$ are here to exclude trivial divisors as the zero divisor or the canonical divisor.

In all of the sequel, we will use the standard notation $g^r_d$ for a linear system of projective dimension $r$ and degree $d$. Hence, the Clifford index is the smallest integer $c$ such that $C$ has a \textit{complete} $g^n_{2n+c}$ for some positive integer $n$. 

Green conjectured that the Clifford index is related to some vanishing properties of the Betti numbers, namely to the linear colength of $C$.
\begin{conjecture}[\textbf{Green's conjecture} \cite{GL}]\label{GreenConj} Assume that $\chara(K)=0$. Let $C$ be a smooth canonical curve. Then $\ell(C)=\Cliff(C)$.
\end{conjecture}

The conjecture does not hold for finite characteristic. The first example provide general curves of genus $g=7$ over a field with $\chara(K)=2$, \cite{S1}.
One direction has been proved by Green and Lazarsfeld.
\begin{theorem}\label{GreenLazarsfeld}\cite[Appendix]{GL}
Inequality $\ell(C)\le \Cliff(C)$ holds over any field.
\end{theorem}
Green's conjecture received a lot of attention. Although it's still open, it is now solved for many curves:

\begin{theorem}\label{Voisin} Assume that $\chara(K)=0$. Then equality $\ell(C)=\Cliff(C)$ holds for:
\begin{enumerate}
\item General curves \cite{V2,V3}.
\item General $d$-gonal curves of any gonality \cite{Ap}.
\item Curves of odd genus with $\ell(C)=(g-1)/2$ \cite{HR,V3}.
\item Smooth curves on any $K_3$ surface \cite{AF}.
\item Curves of linear colenght $\ell(C)\le 2$ \cite{Noet,Petri,S2}. This holds for any characteristic, except if $\chara(K)=2$ and $g= 7$.
\end{enumerate}
\end{theorem}
%
%

It turns out that in most cases, the Clifford index of a curve is computed by a complete linear system of projective dimension $1$, that is by a gonal map. In any cases, the gonality can not be too big when compared to the Clifford index:
\begin{theorem}\label{GonalityClifford}\cite{CM}
We have $\Cliff(C)+2\le \gon(C) \le \Cliff(C)+3$.
\end{theorem}
Moreover, equality $\gon(C) = \Cliff(C)+3$ if and only if $C$ is isomorphic to a smooth plane curve, a complete intersection of cubics in $\pp^3$ or in some other ``extremely rare'' cases (see \cite{ELMS} for a conjectural classification of such curves). Except for these exceptional cases, if Green's conjecture is true, then we can deduce $\gon(C)=\ell(C)+2$ directly from the shape of the Betti table.

\subsection{Brill-Noether theory}\label{ss2.3}

Brill-Noether theory studies the configurations of linear series on smooth curves. All the results we collect here can be found in \cite[Chapters IV and V]{ACGH}. The main object of interest is the subscheme of the variety $\Pic^d(C)$ of degree $d$ line bundles set-theoretically defined as
$$
W^r_d(C) = \{ L \in \Pic^d(C) \mid h^0L \ge r+1 \},
$$
parametrizing complete linear series of degree $d$ and projective dimension at least $r$. The scheme structure of this set comes from its determinantal representation attached to a  homomorphism of suitable vector bundles over $\Pic^d(C)$.  By a direct computational argument, this representation shows in particular that if $W^r_d(C)$ is not empty, then the expected dimension of any of its components is at least
$$
\rho(g,r,d):=g-(r+1)(g-d+r),
$$
the so-called \textit{Brill-Noether number}. Note that the bound on the dimension is not always sharp. In particular, there might exist some $g^r_d$'s when $\rho<0$: hyperelliptic curves of genus $\ge 3$, trigonal curves of genus $\ge 5$, tetragonal curves of genus $\ge 7$, are such examples.

\begin{theorem}\citep{ACGH} Let $C$ be a smooth curve of genus $g$. Let $d\ge 1$ and $r\ge 0$ be integers. 
\begin{enumerate}
\item If $\rho\ge 0$, then $W^r_d(C)$ is non empty of dimension at least $\rho$ at each point. It is connected if $\rho\ge 1$.
\item If $C$ is a general curve, then $W^r_d(C)$ has dimension $\rho$, empty if $\rho<0$, and smooth outside $W^{r+1}_d(C)$. It is irreducible if $\rho\ge 1$.
\end{enumerate}
\end{theorem} 

Of particular interest for us is the case $r=1$. We have 
$$
\rho(g,1,d) \ge 0 \iff g\le 2d-2,
$$
hence the bound 
$$
\gon(C)\le \lfloor \frac{g+3}{2}\rfloor,
$$
which is reached for general curves. Let $\mathcal{M}_g$ denotes the moduli of genus $g$ curves, and let $\mathcal{M}_g^d\subset \mathcal{M}_g$ denotes the subvariety whose points are curves with at least one $g^1_d$. Hence, we have a stratification 
$$
\mathcal{M}_g^2\subset \mathcal{M}_g^3 \subset \cdots \subset \mathcal{M}_g^{\lfloor \frac{g+3}{2}\rfloor}=\mathcal{M}_g.
$$
Each gonality strata $\mathcal{M}_g^d$ is an irreducible subvariety of dimension $2g-5+2d$ \cite{AC}. Therefore, it makes perfect sense to speak about a general point in a given gonality stratum, corresponding to a general $d$-gonal curve. The following result is well known.

\begin{theorem}\label{W1d}
Let $C$ be a general $d$-gonal curve of genus $g$. 
\begin{enumerate}
\item \cite{Ke85} If $d=\lfloor \frac{g+3}{2}\rfloor$ and $g$ is even, then $W^1_d(C)$ is finite and reduced, of cardinality
$$
\frac{g!}{(g-d+1)!(g-d+2)!}
$$
\item \cite{Ke85,EH87} If $d=\lfloor \frac{g+3}{2}\rfloor$ and $g=2n+1$ is odd, then $W=W^1_d(C)$ is an irreducible smooth curve of class
$$
\frac{g!}{(g-d+1)!(g-d+2)!}\Theta^{g-1}
$$
in the cohomology ring of the Jacobian of $C$, where $\Theta$ is the class of the theta-divisor.The genus of $W$ is
$$
 g(W) = \frac{2g!}{(n-1)!(n+2)!}+1=2 {2n+1\choose n-1}+1
$$

\item \cite{Kim2} If $d<\lfloor \frac{g+3}{2}\rfloor$ is not maximal, then $W^1_d(C)$ is a reduced point (the gonal map is unique).
\end{enumerate}
\end{theorem}

\begin{remark}
There are deep relations beetween the dimension of $W^r_d(C)$ for special linear series and the gonality. As a basic example, if $C$ is a smooth curve of genus $g\ge 3$ and $d\le g+r-1$, $r\ge 1$, then $\dim W^r_d(C)\le 2d-2$ with equality if and only if $C$ is hyperelliptic. See for instance \cite{Kim} and the references therein for results in that direction.
\end{remark}

\section{Gonal map, scrolls and scrollar syzygies}\label{S3}

Let $C\subset \pp^{g-1}$ be a smooth irreducible canonical curve of genus $g\ge 3$ given by its homogeneous ideal 
$$
I_C\subset R:=k[x_0,\ldots,x_{g-1}]
$$ 
Let us denote by $d$ the gonality of $C$ and let
$$
f:C\stackrel{d:1}\longrightarrow \pp^1
$$
be a gonal map of $C$. Let us consider the variety set theoretically definedby
$$
X:=\bigcup_{\lambda \in \pp^1} \overline{D_{\lambda}}
$$
with $D_{\lambda}:=f^{-1}(\lambda)$, and $\overline{D_{\lambda}}$ being the smallest projective subspace containing $D_{\lambda}$ (the \textit{linear span} of $D_{\lambda}$). The variety $X$ is a fibration over $\pp^1$ whose fibers are projective spaces. Such varieties are called {\it scrolls}. We have a commutative diagram
\begin{equation}\label{eval}
\begin{array}{ccccc}
 &C\,\, & \subset & & X \\
 & d:1 \searrow & \quad &  \swarrow  & \\
 &  & \pp^1 & & 
\end{array}
\end{equation}
the projection map $X \to \pp^1$ being a morphism unless $X$ is a cone. Computing the gonal map $f$ is thus reduced to compute the homogeneous ideal $I_X$ and the projection $X\to \pp^1$.

\subsection{Scrolls.}\label{ss3.1}

We refer the reader to \cite[Appendix A2H]{Eis} for an introduction to scrolls. 

\begin{lemma}
The scroll $X$ defined before is a non degenerated rational and normal subvariety of $\pp^{g-1}$ of dimension $d-1$ and degree $g-d$.
\end{lemma}

\begin{proof} Let $D$ be an effective divisor of $C$ and let $V\subset K^g$ so that $\overline{D}=\pp(V)$.  The projective space $\pp(K^g/V)^*$ may be identified with the set of projective hyperplanes containing $D$, that is with the complete linear system $|\omega_C(-D)|$. It follows that
$$
\dim |\omega_C(-D)| = g-2-\dim \overline{D}.
$$
Combined with the Riemann-Roch equality $\dim |\omega_C(-D)|+\dim |D| = deg(D) - g +1$, we obtain 
$$
\dim \overline{D}= -\dim |D|+deg(D) -1,
$$
the so-called geometric Riemann-Roch theorem. The pencil $f:C\to \pp^1$ is a complete linear system, since otherwise  you can produce base points and find a pencil with smaller degree contradicting the gonality assumption. Hence, $\dim |D_{\lambda}|=1$ and by the geometric Riemann-Roch theorem we obtain
$$
\dim \overline{D_{\lambda}}=d-2.
$$
The scroll $X$ is thus a $(d-1)$-dimensional variety which is a fibration over $\pp^1$ whose fibers $\simeq \pp^{d-2}$ cut out the linear system $|D_{\lambda}|$ on $C$.   The scroll has degree $g-d$ by \cite[Theorem A2.64]{Eis}.
\end{proof}

\paragraph*{Scrolls from $1$-generic matrices.} Consider $f^* \sO_{\PP^1}(1) \cong \sO_C(D)$. We have $h^0(\sO_C(D))=2$ and Riemann-Roch gives $h^0(\omega_C(-D))=g-d+1$.
Thus the multiplication tensor 
$$
\h^0\sO_C(D) \tensor \h^0 \omega_C(-D) \to \h^0 \omega_C \cong \h^0 \sO_{\PP^{g-1}}(1)
$$
defines a $2 \times (g-d+1)$ matrix $\varphi$ of linear forms. More precisely, the $(i,j)$-entry of $\varphi$ is the linear form $l_{ij}:=x_i\otimes y_j$, where $x_i$ and $y_j$ respectively run over a basis of $\h^0\sO_C(D)$ and $\h^0 \omega_C(-D)$. The matrix $\varphi$ is \textit{$1$-generic} \cite[Proposition 6.10]{Eis}, meaning that any of its \textit{generalized rows} (non zero linear combination of rows) consits of linearly independant linear forms. The  $2\times 2$ minors of $\varphi$ define a homogeneous ideal of quadratic forms which all vanish on the curve $C$. This ideal coincides with the homogeneous ideal $I_X$ of the rational normal scroll $X$ previously defined \cite[Corollary 3.12]{Eis}. The quotient of the two entries of any column of $\varphi$ defines a rational function $f\in K(\pp^{g-1})$ which induces the map $f:X\to \pp^1$ we are looking for.

The minimal free resolution of $S/I_X$ thus coincides with the Eagon-Northscott complex $(E_{\bullet})$ of the matrix $\varphi$, this complex being a subcomplex of the minimal free resolution $(F_{\bullet})$ of $S_C$ \cite[Proposition 6.13]{Eis}. In particular, it follows that $X$ is arithmetically Macaulay, and we recover that $X$ has codimension $g-d-1$ and minimal degree $g-d$ \cite[Section A2H]{Eis}.

\paragraph*{Structure of scrolls.} The $d:1$ morphism $f:C\rightarrow \pp^1$ induces a rank $d$-vector bundle $f_*(\omega_C)$ on $\pp^1$ which, by Grothendieck's theorem, splits as a direct sum of line bundles
$$
f_*(\omega_C)=\omega_{\pp^1}\oplus\mathcal{O}_{\pp^1}(e_1)\oplus\cdots \oplus \mathcal{O}_{\pp^1}(e_{d-1}),
$$
with $e_1+\cdots+e_{d-1}=g-d$. This determines a vector bundle
$$
\mathcal{E}:=f_*(\omega_C)/\omega_{\pp^1}\simeq \mathcal{O}_{\pp^1}(e_1)\oplus\cdots \oplus \mathcal{O}_{\pp^1}(e_{d-1})
$$
of rank $(d-1)$ on $\pp^1$ and a projective bundle  
$$
\pi:\pp(\mathcal{E})\to \pp^1.
$$
The image
$$
j:\pp(\mathcal{E})\rightarrow X\subset \pp \h^0(\mathcal{O}_{\pp(\mathcal{E})}(1))=\pp^{(e_1-1)+\cdots+(e_{d-1}-1)}=\pp^{g-1}
$$
is a rational normal scroll $X$ of codimension $g-d$ and minimal degree $g-d$ that coincides with the scroll defined before. It is a cone if some of the $e_i$'s are zero, and in this case the map $\pp(\mathcal{E})\rightarrow X$ is a small desingularization. Otherwise $\pp(\mathcal{E})\cong X$. The Picard group of $\pp(\mathcal{E})$ is the free $\zp$-module
$$
\Pic(\pp(\mathcal{E}))\simeq H\zp\oplus R\zp
$$
generated by the hyperplane section class $H=[j^*\mathcal{O}_{\pp^{g-1}}(1)]$ and by the ruling class $R=[\pi^*\mathcal{O}_{\pp^1}(1)]$. 

The $1$-generic matrix $\varphi$ of linear forms defined above coincides with the matrix of the multiplication map
$$
\h^0\mathcal{O}_{\pp(\mathcal{E})}(R)\otimes \h^0 \mathcal{O}_{\pp(\mathcal{E})}(H-R)\longrightarrow \h^0 \mathcal{O}_{\pp(\mathcal{E})}(H)
$$
under the isomorphisms $\h^0\mathcal{O}_{\pp(\mathcal{E})}(R)\simeq \h^0\sO_C(D)$ and $\h^0\mathcal{O}_{\pp(\mathcal{E})}(R)\simeq \h^0\omega_C$. 

\subsection{Scrollar syzygies}\label{ss3.2}

We study now the relations beetween the Eagon-Northscott complex of the scroll $X$ and the projective resolution of $C$, following Bothmer's approach \cite{both,both2}. All what follows in this subsection remains valid if we replace $C$ with any irreducible non degenerate projective variety.

Let $S_C=S/I_C$ be the homogenous coordinate ring of $C$ and let
$(F_{\bullet})$ be its minimal free resolution. The \textit{linear strand} of $(F_{\bullet})$ is the subcomplex $L_{\bullet}\subset F_{\bullet}$ 
$$
(L_{\bullet}): \qquad S \leftarrow S(-2)^{\beta_{1,2}}\leftarrow S(-3)^{\beta_{2,3}} \leftarrow S(-4)^{\beta_{3,4}}\leftarrow \cdots
$$
Note that the linear strand consists of linear syzygies, except for the first differential $L_0\leftarrow L_1$ of quadratic equations. Suppose that $L_{2}\ne 0$ and let $s\in L_p$ be a $p^{th}$ linear syzygy for some $p\ge 2$. Let $V=V_s$ be the smallest $K$-vector space such that the following diagram commutes
\begin{equation}\label{diagram}
\begin{array}{ccc}
L_{p-1} & \leftarrow & L_{p} \\
\cup    &  	      &   \cup  \\
V\otimes S(-p) & \leftarrow  &S(-p-1)  \cong  \langle s \rangle
\end{array}
\end{equation}
We define the {\it rank} of $s$ to be the dimension of $V$. This diagram extends to a map from the Koszul complex of $V$ to the linear strand of $C$. Namely, $\Hom(L_{\bullet},S)$ is a free complex and the dual Koszul complex is an exact complex, so the maps of the dual diagram extend to a morphism of complexes, which we dualize again. 
\begin{equation*}
\begin{array}{cccccccccccc}
& S\quad &\leftarrow  & L_1 & \leftarrow   \cdots   \leftarrow & L_{p-1} & \leftarrow & L_{p} \\
&   \uparrow \varphi_2 &     &      \uparrow   &                                       & \uparrow    &  	      &   \uparrow  \\
\wedge^{p+2} V \otimes S \leftarrow  & \wedge^{p+1} V \otimes S(-1)& \leftarrow & \wedge^p V \otimes S(-2) & \leftarrow   \cdots  \leftarrow &  V\otimes S(-p) & \stackrel{\varphi_1^t}\leftarrow  & S(-p-1)
\end{array}
\end{equation*}
Note that all vertical maps but the last map $\varphi_2$ are unique because there are no non trivial homotopies for degree reasons. We call the \textit{syzygy scheme} of $s$ the subscheme $Syz(s)\subset \pp^{g-1}$ defined by the ideal of quadratic forms
$$
I_s:=Im(\wedge^p V \otimes S(-2)\longrightarrow S).
$$
Note that the syzygy scheme contains the curve $C$ by the commutativity of the diagram. We can compute the syzygy scheme in some particular cases. 

\begin{definition}\cite{both}
A linear syzygy $s\in L_p$ is \textit{a scrollar syzygy} if it has rank $p+2$. 
\end{definition}

If $C$ is irreducible and non degenerate (which is the case of canonical curves), any linear syzygy $s\in L_p$ has $\rank(s)\ge p+2$ \cite{both2} and a scrollar syzygy has the minimal possible rank $p+2$. The following result explains the \textit{scrollar} terminology. 

\begin{theorem}\cite{both2}\label{Bothmer}
The syzygy scheme $Syz(s)$ of a scrollar syzygy $s\in L_p$ is a scroll of degree $p+2$ and codimension $p+1$ that contains $C$. 
\end{theorem}

Let us explain this key result. A scrollar syzygy $s\in L_p$ induces a map 
$$
S(-p-1) \stackrel{\varphi_1^t}\longleftarrow   V \otimes S(-p)
$$
Since $\rank(s)=\dim(V)=p+2$, and by duality of the Koszul complex,
the map $\varphi_1$ is dual to the left end side map
$$
0\longleftarrow S \stackrel{\varphi_1}\longleftarrow   \wedge^{p+1} V \otimes S(-1) \stackrel{\phi} \longleftarrow  \wedge^p V \otimes S(-2)
$$
In a suitable basis of $V$, the $\binom{p+2}{2}$ columns of the syzygy matrix $\phi$  have $j^{th}$ entry $l_i$ and $i^{th}$ entry $-l_j$ and other entries zero, where  $\phi_1=(l_0,\ldots,l_{p+1})$.

The vertical map 
$$
\varphi_2: \wedge^{p+1}V \otimes S(-1) \to S
$$
gives another row of linear forms $(m_0,\ldots,m_{p+1})$ in the fixed basis of $V$. The composition 
$$
\varphi_2\circ \phi : \wedge^p V \otimes S(-2)\longrightarrow S
$$
has entries the $2\times 2$ minors of the matrix 
$$
\varphi:=\begin{pmatrix}
l_0 & l_1 & \cdots & l_{p+1} \\ 
m_0 & m_1 & \cdots & m_{p+1}
\end{pmatrix}
$$
with upper row $\varphi_1$ and lower row $\varphi_2$ (all the minors occur). The ideal $I_s$ of the syzygy scheme $Syz(s)$ is generated by these minors. Since $C$ is irreducible, $I_s$ contains no reducible quadrics. So the matrix $\varphi$ is 1-generic, and $I_s$ is the homogeneous ideal of a codimension $p+1$ rational normal scroll $X$ \cite[Chapter 6B]{Eis}. A projection $X\to \pp^1$ is given by the ratio of the entries of a column of $\varphi$.

\subsection{Proof of Theorem 1: Algorithms and examples}\label{ss3.2}

\begin{proposition}\label{ScrollarSyzygy}
Let $C$ be a smooth canonical curve of genus $g$ and linear colength $\ell$.  Then $\gon(C)=g-p+2$, where $p$ is the  biggest such integer $\le g-\ell$ for which there exists a $p^{th}$ scrollar syzygy.
\end{proposition}

\begin{proof}
Since $\gon(C)\ge \ell+2$ by Green-Lazarsfeld's theorem, $C$ is contained in a scroll of codimension $\le g-\ell$, hence admits a $p^{th}$ scrollar syzygies for some $p\le g-\ell$. Let $p$ be the biggest integer $\le g-\ell$ such that there exists a $p^{th}$ scrollar syzygy. Such a syzygy gives rise to a codimension $p+1$ scroll $X$ that contains $C$. The rulings of $X$ cuts out a $g^1_d$ on $C$ for some integer $d$. We claim that this pencil is necessarily complete and base point free. If the $g^1_d$ has a base point $x\in C$, projection from $x$ would lead to a pencil of smaller degree. If the $g^1_d$ is not complete, we can impose a base point and produce a pencil of smaller degree. In both case, $\gon(C)<d$ and $C$ would be contained in a scroll of codimension $>p+2$ contradicting our maximality assumption on $p$. Hence the $g^1_d$ is complete and $\gon(C)=d$. The relation $\gon(C)=g-p+2$ follows from the geometric version of Riemann-Roch. 
\end{proof}

\begin{remark}
In general, the $g^1_d$ induced by a scrollar syzygy  is base point free if $X$ is not a cone, or if $C$ does not pass through the vertex of the cone. It is a complete linear series if the rulings $\pp^{g-p}$ defined by any generalized row (non zero linear combination of the rows) of $\varphi$ cut $C$ in a set of points which span this space. 
\end{remark}

\begin{remark}
If Green's conjecture is true, and since $\Cliff(C)+2\le \gon(C)\le \Cliff(C)+3$, in order to compute a gonal map we have to look for scrollar syzygies for only two possible values $p=g-\ell$ and $p=g-\ell-1$, the second case corresponding to those exceptional curves of Clifford dimension $>1$. In any case, we have that $\gon(C)\le \lfloor (g+3)/2\rfloor$ giving the lower bound $p\ge \lfloor (g+2)/2\rfloor$.
\end{remark}

It is now clear how to compute a gonal map on $C$: We first compute a scrollar syzygy $s$ of highest homological degree $p$ and then we compute the gonal map $C\to \pp^1$ it induces. 

Let us first detail the second step.

\begin{algo}[Gonal map from a scrollar syzygy]\label{alg1} \noindent

\textbf{Input:} A scrollar syzygy $s$ of highest homological degree $p$ defined over a field $K$.

\textbf{Output:} A gonal map defined over $K$.

\textit{Step 1.} Compute the matrix $\varphi$ in a fixed basis of $V$: Computing the first row $\varphi_1$ is clear. To compute $\varphi_2$, we dualize the diagram \ref{diagram}, and we complete it as a map of complex thanks to the projectivity of $\Hom(L_{\bullet},S)$ and the exactness of $\Hom(V_{\bullet},S)$. Dualizing again, we get the map $\varphi_2$ as the left end vertical map.

\textit{Step 2.} Return the rational function $f\in K(C)$ given by the restriction to $C$ of the ratio of the first entries of $\varphi_1$ and $\varphi_2$.
\end{algo}

Algorithm \ref{alg1} terminates and returns the correct answer thanks to Proposition \ref{ScrollarSyzygy} and the discussion before.

\paragraph*{Space of scrollar syzygies.} Let $\psi_p:L_p\to L_{p-1}$ be the syzygy matrix, with linear entries in the indeterminates $x=(x_0,\ldots,x_{g-1})$. Let $y=(y_0,\ldots, y_{\beta-1})$ be new indeterminates that represent a formal combination $s=y_0s_0+\cdots+y_{\beta-1}s_{\beta-1}$ of a basis $(s_0,\ldots,s_{\beta-1})$ of $L_p$. 
We can write 
$$
\psi_p(x)\begin{pmatrix}
y_0 \\ 
\vdots \\ 
y_{\beta-1}
\end{pmatrix} =\Psi_p(y)\begin{pmatrix}
x_0 \\ 
\vdots \\ 
x_{g-1}
\end{pmatrix}
$$
for some matrix $\Psi_p(y)$ with coefficients linear forms in $y$. Then $\rank(s)=\rank(\Psi_p(y))$. Hence the following definition:

\begin{definition}\label{defScrollar}
The \textit{subspace of $p^{th}$ scrollar syzygies} of $C$
is the determinantal subscheme 
$$
Y_{p}\subset \pp(\Tor^S_p(S_C,K)_{p+1})\cong \pp^{\beta-1}
$$ 
whose homogeneous ideal is generated by the $(p+2)\times (p+2)$ minors of the matrix $\Psi_p(y)$. We will abusively say \textit{the space of scrollar syzygies} of $C$, denoted by $Y$, for the subspace of scrollar syzygies of highest homological degree $p$.
\end{definition}

\begin{algo}[Scrollar syzygies]\label{alg2} \noindent

\textbf{Input:} A smooth canonical curve $C$ defined over a field $K$.

\textbf{Output:} A scrollar syzygy of highest homological degree defined over some field extension of $K$.

\textit{Step 1.} Compute the linear strand of $S_C=S/I_C$. Deduce the linear colength $\ell$ of $C$.

\textit{Step 2.} Let $p:=g-\ell$. While $p\ge [(g-2)/2]$ do:

2.a. Compute the matrix $\Psi_p(y)$ defined above.

2.b. Check if the space $Y_{p}$ of scrollar syzygies defined by the $(p+2)\times (p+2)$-minors of $\Psi_p(y)$ is empty (membership ideal problem).

2.c. If $Y_{p}$ is empty do $p:=p-1$. Else go to Step 3.

\textit{Step 3.} Compute a point $s\in Y_{C,p}$, using possibly a field extension. Return $s$.
\end{algo}

The algorithm terminates thanks to Proposition \ref{ScrollarSyzygy}. 

Combinning these results, we get the following complete algorithm for computing gonal maps of projective curves, completing the proof of Theorem \ref{t1}:

\begin{algo}[Gonal map of projective curves]\label{alg3} \noindent

\textbf{Input:} The homogeneous ideal of an absolutely irreducible projective curve $C'$.

\textbf{Output:} The gonality of $C'$ and a rational map $f:C'\to \pp^1$ of minimal degree.

\textit{Step 1.} Check if $C'$ is rational or hyperelliptic. If yes, return a gonal map. 

\textit{Step 2.} Compute the homogeneous coordinate ring $S_C$ of a smooth canonical model $C$ of $C'$. 

\textit{Step 3.} Call Algorithm \ref{alg2} to compute a scrollar syzygy $s$ defined over some field extension $L$ of $K$.

\textit{Step 4.} Call the Algorithm \ref{alg1} with input $s$ to obtain the gonal map $f:C\to \pp^1$.

\textit{Step 5.} Return $f$ composed with the rational map $C'\rightarrow C$ of step $1$. 
\end{algo}

\begin{example} We compute a $g^1_4$ on a general curve of genus $g=6$.  It iso well-known that the canonical model of a general curve of genus $6$ is the complete intersection 
 $Y \cap Q \subset \pp^5$ of a Del-Pezzo surface $Y$ of degree $5$ and a quadric. $C$ has precisely five $g^1_4$'s. The space of scrollar syzygies is the union of five  disjoint lines in $\pp(\Tor_2^S(S_C,K)_3) \cong \pp^{4}$. 
 For a general ground field $K$ we will need a degree $5$ field extension $L \supset K$ to specify one of the lines, and the $g^1_4$ will be defined over $L$. Over a finite field $K$ with $q$ elements, one of the $g^1_4$'s will be defined already over $K$ in about $63 \%$ of the cases, because polynomials over finite fields have frequently a factor. More precisely the proportion is
 $19/30-(5/6) q^{-1}+O(q^{-2})$, see e.g. \cite{EHS}, Section 2.
\end{example}

\begin{example} To compute a $g^1_5$ on a general curve of genus $g=7$ with this method failed. By Mukai' Theorem, the canonical model of curve $C$ of genus $g=7$ and Clifford index $\Cliff(C)=3$ is a transversal intersection section of the $10$-dimensional spinor variety $S^{10} \subset \pp^{15}$ with a six dimensional linear subspace.
The Brill-Noether loci 
$$
W=W^1_5(C) = \{ L \in Pic^5(C) \mid h^0 L \ge 2 \}
$$
is a curve, which is smooth of genus  $15$, if $C$ is Petri general. The space of scrollar syzygies
$$
Y \subset \pp(\Tor_2^S(S_C,K)_3) \cong \pp^{15}
$$ 
is defined by the $4 \times 4$ minors of a $7 \times 10$ matrix. To find a scrollar syzygy is equivalent to finding a point on $Y$. The surface $Y$ is a $\pp^1$-bundle over $W$. We already failed with the attempt to compute the ideal of minors.
\end{example}

\begin{example} Using the \Mac -package  RandomCurves \href{http://www.math.uni-sb.de/ag/schreyer/home/computeralgebra.htm}{http://www.math.uni-sb.de/ag/schreyer/home/computeralgebra.htm}  we can search for a curve of genus $g=9$ with two $g^1_5$ over small finite prime fields. In this case the space of scrollar syzygies
$$Y \subset \pp(\Tor_4^S(S_C,K)_5) \cong \pp^7$$ is the union of two rational normal curves of degree $3$. The scheme $Y$ is the radical of the  ideal of $6\times 6$ minors of a suitable
$9 \times 70$ matrix $\Psi$ of linear forms on $\pp^7$. This computation is two heavy. However the annihilator of $\coker \Psi$ can be computed. It is the ideal of the  join variety to the two rational normal curves. The computation of the two rational normal curves from the join is possible. The two $g^1_5$'s are both  defined over the finite ground field in about $50 \%$ of the cases.
\end{example}

Unfortunately, Algorithm \ref{alg2} is {\it a priori} too space consuming and has a prohibitive complexity to be used in practice: we have to solve too many determinantal equations of too high degrees to find scrollar syzygies. We couldn't solve this system for general curves of genus $g\ge 7$.

\section{Goneric curves}\label{S4}
We now pay attention to a particular class of curves for which Algorithm \ref{alg2} can be dramatically improved.
\subsection{Definition and examples}\label{ss4.1}
A $d$-gonal smooth canonical curve $C$ being contained in a $(d-1)$-dimensional scroll $X$, we always have inequality
\begin{equation*}\label{ineq}
\beta_{d-2,d}(C) = \beta_{g-d,g-d-1}(C)
 \ge  \beta_{g-d,g-d-1}(X)=g-d,
\end{equation*}
the first equality using symmetry of the Betti table of $C$, the inequality using the inclusion $I_X\subset I_C$ and the last equality using the well known shape of the Eagon-Northscott resolution of scrolls. It's thus natural to introduce the following definition:

\begin{definition}\label{dgoneric}
An irreducible, smooth canonical curve $C$ of linear colength $\ell$, gonality $d$ and genus $g$ is called \textit{goneric} if it satisfies
\begin{enumerate}
\item $d=\ell+2$
\item $\beta_{g-2-\ell,g-1-\ell}=g-2-\ell$.
\end{enumerate}
\end{definition}

\begin{example}
\begin{enumerate}
\item Trigonal curves are goneric.
\item Tetragonal curves with no $g^r_{2+2r}$ for $ 2 \le r \le 3$ are goneric by \cite{S1}.
\item A general $d$-gonal curve of genus $g > (d-1)^2$ is goneric by \cite{S3}.
\item A general curve of genus $g$ is not goneric.
\item Over a ground field of characteristic zero, a general $d$-gonal curve in the biggest gonality strata $\mathcal{M}_{g}^d\subset \mathcal{M}_g$ ($g=2d-1$) is goneric by \cite{HR} and \cite{V3}.
\end{enumerate}
\end{example}

More generally, we believe that general $d$-gonal curves of non maximal gonality are goneric (see Conjecture \ref{conj2}  for a more precise statement). Moreover, we believe that condition $(1)$ in Definition \ref{dgoneric} is superflues that is, gonericity can be read from the Betti table alone in almost all cases.

\begin{conjecture}\label{conj1}
In characteristic zero, an  irreducible, smooth canonical curve $C$  of linear colength $\ell$ is $\ell+2$-goneric if and only if $\beta_{g-2-\ell,g-1-\ell}=g-2-\ell$ unless $(g,\ell)=(6,1)$ and $C$ is isomorphic to a smooth plane quintic.
\end{conjecture}

For goneric curves, Algorithm \ref{alg2} simplifies. The key is the geometry of scrollar syzygies of a scroll.

\begin{proposition}\label{adjOnScroll} Let $X \subset \pp^n$ be a rational normal scroll of codimension $p$, hence degree $p+1$. Then $\dim \Tor_p^S(S_X,K)_{p+1} = p$ and the space of scrollar syzygies
$$
Y_p\subset \pp(\Tor_p^S(S_X,K)_{p+1})\cong \pp^{p-1}
$$ 
is a rational normal curve of degree $p-1$. All syzygies $s \in \pp^{p-1} \setminus Y_p$ have rank
$$\rank s = n - \dim \Sing(X)$$
where by convention the empty set has dimension $-1$.
\end{proposition}

\begin{proof} We first discuss the case when $X$ is smooth. Let $k=n-p-1$ be the dimension of the ruling. Then the canonical bundle on $X$ is
$$\omega_X= \sO_X(-(k+1)H+(p-1)R)$$
where $H$ is the class of an hyperplane section and $R$ is the class of the ruling. Let $E_\bullet \to S_X$ be the Eagon-Northcott complex.
If $\varphi: F \to G$ denotes the $2\times (p+1)$ matrix, whose minors define $X$, then 
$$
E_i =( \Sym_{i-1} G)^* \otimes \wedge^{i+1} F \otimes \wedge^2 G^*
$$
by \cite{Eis}. In particular, $E_p \cong S^p(-p-1)$ and $E_{p-1} \cong S^{(p-1)(p+1)}(-p)$. By the adjunction formula
$$
\omega_X\cong\mathcal{E}xt^{p}(\mathcal{O}_X,\omega_{\pp^n}),
$$
the dual of the last map of the sheafified Eagon-Northscott complex $(\mathcal{E}_{\bullet})$ gives a presentation of $\omega_X$:
$$
0 \leftarrow \omega_X \leftarrow \sHom(\mathcal{E}_{p}, \omega_{\pp^n}) \leftarrow  \sHom(\mathcal{E}_{p-1}, \omega_{\pp^n})
$$
Twisting by $(k+1)H$ we get the presentation
$$
 0 \leftarrow \sO_X((p-1)R) \leftarrow \sO_{\pp^n}^p \stackrel{\psi_p^t}{\longleftarrow} \sO_{\pp^n}^{(p-1)(p+1)}(-1)
$$
Thus $\h^0(\omega_X((k+1)H))\cong \h^0(\sO_{\pp^1}(p-1))$ is $p$-dimensional and the complete linear system $|\omega_X((k+1)H)|$ defines a morphism
$$
X \to Y_p \subset \pp^{p-1}
$$
onto a rational normal curve of degree $p-1$. The graph of this morphism in $\pp^n \times \pp^{p-1}$ is defined by the entries of the matrix
$$
(y_0,\ldots,y_{p-1}) \psi_p^t(x)
$$
in terms of homogeneous coordinates $(x_0,\ldots x_n)$ on $\pp^n$ and $(y_0,\ldots y_{p-1})$ on $\pp^{p-1}$.
The fiber over a point $(a_0,\ldots,a_{p-1})$ is the zero loci of $(a_0,\ldots,a_{p-1}) \psi_p^t(x)$. Thus the syzygy $s=s_a$ corresponding to the point $a$ has either maximal rank $n+1$
if $a \notin Y_p$ or rank $p+2$ if $a \in Y_p$. This completes the proof in the smooth case. In the singular case the singular set is defined by the entries of the matrix $\varphi$, and the Eagon-Northcott complex which involves only these variables. Thus the the rank in this case is $n- \dim \Sing(X)$.
Note that the scrollar syzygies correspond to the compeletely decomposable tensors in $(\Sym_{p-1} G)^* \otimes \wedge^{p+1} F \otimes \wedge^2 G^*$.
\end{proof}

\begin{remark}[\textbf{Parametrization of rational normal curves}]\label{param} In case $X \subset \pp^n$ is a rational normal curve, the proposition above produces a rational normal curve of degree $n-2$  in $ \pp^{n-2}$, and iterating we finally obtain a line or a conic, depending on the parity of $n$.
If we now parametrize the line or conic, we can compute a parametrization of the original rational normal curve, by an iterated syzygy computations.
The advantage to use iterated adjunction, is that we do not have to make any choices.
\end{remark}

\begin{algo}[Gonal maps of goneric curves] \label{alg4} \noindent

\textbf{Input:} A smooth  goneric curve $C$ of genus $g$ in its canonical embedding.

\textbf{Output:} A gonal map.

\textit{Step 1.} Compute the linear strand $L_\bullet$ of the resolution. Let $p=g-2-\ell$ be the length and let $\psi_p$ be the last differential. By assumption $\psi_p^t=\psi_p^t(x)$ is a $p  \times (p-1)(p+1)$ matrix with linear entries in $x$. If this is not the case return \textit{``$C$ is not goneric"}.

\textit{Step 2.} Write the product $(y_0,\ldots y_{p-1}) \psi_p^t(x)$ in the form
$$
(y_0,\ldots y_{p-1}) \psi_p^t(x)= (x_0,\ldots, x_{g-1}) \Psi_p(y).
$$
Then $\Psi_p=\Psi_p(y)$ is a matrix of size $g\times (p-1)(p+1)$ with linear entries in $y$ and 
$$
\coker(\Psi_p: \sO_{\pp^{p-1}}(-1)^{(p-1)(p+1)} \rightarrow \sO_{\pp^{p-1}}^g)
$$
 has support on a rational normal curve $Y_p$  of scrollar syzygies.

\textit{Step 3.} Compute the annihilator of $J=\coker \Psi_p$, the defining ideal of $Y_p$. If the  zero loci $V(J)$ has dimension different from $1$ or degree different from $p-1$
then return \textit{``$C$ is a smooth plane quintic"}  in case $(g,\ell) =(6,1)$, otherwise return \textit{``$C$ is a counter example to Conjecture \ref{conj1}"}.

\textit{Step 4.} Compute a point $s \in Y_p$ on the rational normal curve of scrollar syzygies.

\textit{Step 5.} Compute the induced gonal map using Algorithm \ref{alg1}.

\end{algo}

\begin{remark}\label{pointOnCurve} Step 4 can be solved in various ways. For example, we can use Remark \ref{param}, and, if necessary, any method which parametrizes conics. So we might need a quadratic field extension in case $p=g-2-\ell \equiv 1 \mod 2$. In case of a finite ground field $K$, we may use a probabilistic method to find a point on $Y_p$, by say decomposing the intersection
$Y_p \cap H$ of $Y_p$ with a random hyperplane section into $K$-rational components.
\end{remark}

\begin{example} Using a program generating random curves of genus up to 14 (reference), we created random curves of genus 9 over a finite field of characteristic not equal to $3$ with $q$ elements. After roughly $q$ trials, we obtained a non general curve with Betti table
\[ \begin{array}{cccccccc}
  1 &  - &  - &  - &  - & -  &  - & - \\
  - & 21 & 64 & 70 &  4 &  - &  - & - \\
  - &  - &  - &  4 & 70 & 64 & 21 & - \\
  - &  - &  - &  - &  - &  - &  - & 1 \end{array}  \]
This curves obeys to $\beta_{\ell,\ell+2}=g-\ell-2$. Our algorithm shows that the curve is $5$-goneric and produces a $4$-dimensional scroll $X$ with Betti table
\[ \begin{array}{ccccc}
  1 &  - &  - &  - &  - \\
  - & 10 & 20 & 15 &  4. \end{array}  \]
that induces a $g^1_5$. Note that in this case, we could have concluded that $C$ is $5$-goneric from its Betti table alone thanks to \cite{Sag}.
  
\end{example}

\subsection{Geometric characterization of goneric curves}

\begin{conjecture}\label{conj2} For ground field of characteristic zero, a non hyperelliptic curve $C$ is $d$-goneric if and only if
\begin{enumerate}
\item $W^1_d(C)$ is a single reduced point
\item this point is the unique line bundle of degree $\le g-1$ which computes the Clifford index of $C$.
\end{enumerate}
\end{conjecture}

The second condition is necessary, since for a example curves $C$ which have a plane model degree $d+2$ with a single node satisfy (1), but $\beta_{p,p+1}(C) > p$.
One direction of this conjecture is easy:

\begin{proposition} Let $C$ be a goneric non hyperelliptic curve. Then
\begin{enumerate}
\item $W^1_d(C)$ is a single reduced point
\item this point is the unique line bundle of degree $\le g-1$ which computes the Clifford index of $C$.
\end{enumerate}
\end{proposition}

\begin{proof} Let $X$ be the scroll of the gonal line bundle $\sO_C(D)$, with $\codim(X)=p$. Since $\beta_{p,p+1}(X)=\beta_{p,p+1}(C)$ by assumption, there is no further scrollar syzygies in $L_p$ by Proposition \ref{adjOnScroll}. Hence
$W^1_d(C)$ is a single point.  To prove that it is a reduced point, we consider the tangent space of $W^1_d(C)$ at $\mathcal{L}=\sO_C(D)$
$$
\T_{\mathcal{L}} W^1_d(C) ={\rm Im} (\h^0\mathcal{L} \otimes \h^0(\mathcal{L}^{\vee}\otimes\omega_C)\to \h^0\omega_C)^\perp \subset \h^1(\sO_C) = \T_{\mathcal{L}} \Pic(C).
$$
So this map is surjective, if and only if the scroll of $D$ is not a cone, equivalently $h^0\sO_C(2D) =3$. On the other hand, if $h^0\sO_C(2D) \ge 4$ then by \cite{Cop}, $C$ deforms to a curve $C_t$ which has at least two different $g^1_d$'s. In particular $\beta_{p,p+1}(C_t) > p$ and by semi-continuity of Betti numbers  $\beta_{p,p+1}(C) >p$ as well, contradicting the gonericity of $C$. So $W^1_d(C)$ is a reduced point and $X$ a smooth scroll. By Proposition \ref{adjOnScroll} once more we see that
syzygies $s \in L_p$ either have rank $p+2$ or maximal possible rank $g$. Hence there is no  room for further Green-Lazarfeld syzygies \cite{GL}, since those syzygies have rank
$\le p+3$.
\end{proof}

\subsection{Goneric curves with $\ell(C) << g$ : the Lie algebra method}

We saw how to compute the matrix of the scroll from the last map of the linear strand. For curves with small linear colength and high genus, computing the linear strand might be a very hard task, the middle Betti numbers being doubly exponential in $g$. We give here another method that only requires to compute the resolution $(F_{\bullet})$ up to the homological degree $\ell(C)+1$. We get in such a way the homogeneous ideal of the scroll but not its $1$-generic matrix. The computation of the projection $X\to \pp^1$ is no more trivial and we use instead a Lie algebra method.

\subsubsection{Compute the ideal of the scroll by adjunction.}\label{4.3.1} Each syzygy module of $C$ is a direct sum
$$
F_i=L_i\oplus Q_i:=S(-i-1)^{\beta_{i,i+1}}\oplus S(-i-2)^{\beta_{i,i+2}}.
$$
and the matrix $M_i:F_{i}\to F_{i-1}$ can be decomposed as 
$$
M_i=\left(\begin{matrix}
A_{i} & C_i \cr
0 & B_i
\end{matrix}
\right)
$$
where $A_{i}:L_{i}\rightarrow L_{i-1}$ and $B_i: Q_{i}\rightarrow Q_{i-1}$ have linear entries and $C_i:Q_{i}\longrightarrow L_{i-1}$ has quadratic entries (the map $L_{i}\rightarrow Q_{i-1}$ is zero by minimality of the resolution). Note that the matrix $A_i$ coincides with the matrix $\psi_i$ above and $B_i$ is equivalent to the matrix $\psi_{g-i-1}^t$ by the symmetry of the Betti numbers.

\begin{proposition}\label{pgoneric}
Let $C$ be a goneric curve of linear colength $\ell$. The homogeneous ideal of the unique scroll inducing the gonal map is equal to $I_X=Ann(Coker(B_{\ell+1}))$.
\end{proposition}

\begin{proof}
By adjunction, the twisted dual complex $\mathcal{E}_{g-d-\bullet}^{\vee}(-g)\cong\mathcal{H}om(\mathcal{E}_{\bullet},\omega_{\pp^{g-1}})$ of the sheafified Eagon-Norscott complex of $X$ is a free resolution of 
$\omega_X$, hence
$$
\mathcal{I}_X=\Ann(\omega_X) = \Ann\big(\coker(\mathcal{E}_{g-d-1}^{\vee}\to \mathcal{E}_{g-d}^{\vee})\big).
$$
Since $X$ is arithmetically Cohen-Macaulay, taking global sections is exact and it follows that
$$
I_X=\Ann(\coker(E_{g-d-1}^{\vee}\to E_{g-d}^{\vee})).
$$
Since $C\subset X$ is goneric, its last linear Betti number coincides with that of the scroll, so that $E_{g-d}\cong L_{g-d}$. Hence the commutative diagrams
\begin{equation*}
\begin{array}{ccccccc}
L_{g-d-1} & \leftarrow & L_{g-d} &                           & 	L_{g-d-1}^{\vee} 	   &  \rightarrow   &  L_{g-d}^{\vee}  \\ 
\cup    &  	      &   ||    &\qquad \stackrel{duality}\rightsquigarrow	\qquad & \downarrow &  	& 	|| \\ 
E_{g-d-1} & \leftarrow  & E_{g-d}   &                         &  E_{g-d-1}^{\vee}      & \rightarrow    & E_{g-d}^{\vee}
\end{array}
\end{equation*}
The second diagram forces 
$$
\coker(E_{g-d-1}^{\vee}\to E_{g-d}^{\vee})=\coker(L_{g-d-1}^{\vee}\to L_{g-d}^{\vee})=\coker(Q_{\ell+1}\to Q_{\ell}),
$$
last equality using $d=\ell+2$ and self-duality of $(F_{\bullet})$.
\end{proof}

We need now a method to compute the projection map $X\to \pp^1$ from the homogeneous ideal of $X$.

\subsubsection{The Lie Algebra method}

The Lie algebra method is a relatively simple method for constructive recognition of embedded projective varieties with toric actions over a field $K$ of characteristic zero. It makes sense for a variety $X\subset\PP^n$, $n>0$, such that there is an infinite group of projective automorphisms that leave $X$ fixed. The Lie algebra of this group is
\[ L(X) := \{ M \mid \forall F\in I_X:\langle\nabla(F),M(x_0,\dots,x_n)^t\rangle \in I_X \mbox{ and } \mathit{trace}(M)=0\}  \]
(it suffices to check a generating set of $F$'s). In order to compute $L(X)$, we resolve a system of linear equations in $(n+1)^2$ variables.

In the case $X$ is a scroll, the Levi subalgebra of $L(X)$ has a direct summand $L'$ isomorphic to $\mathfrak{sl}_2$, and the space of linear forms in $x_0,\dots,x_n$ inherits the structure of an $\mathfrak{sl}_2$-module. It is a direct product of $d-1$ irreducible modules of dimension $e_1+1,\dots,e_{d-1}+1$, where $d-1$ is the dimension of the scroll and $e_1,\dots,e_{d-1}\ge 0$ are the integers determining the scroll. Each irreducible module has a canonical basis by weight vectors, which is easy to compute by eigenvector computations. Finally, the quotient of two suitable forms in this basis defines the structure map to $\PP^1$.

If $K$ is not algebraically closed, then $L'$ may be a nontrivial twist of $\mathfrak{sl}_2$, and in general one needs an algebraic field extension of order two to compute a Chevalley basis. However, if at least one of the numbers $e_1,\dots,e_{d-1}$ is odd, then $L'$ has a representation of even degree and is therefore $K$-isomorphic to $\mathfrak{sl}_2$. The Levi decomposition is possible without restriction on $K$, but one should mention that the Levi subalgebra is only unique up to inner automorphisms.

Here are two examples for two curves that have been computed by \textit{Magma}. The gonal maps are easy to see without computation, the examples should merely illustrate the method. See \href{http://www.risc.jku.at/people/jschicho/pub/gonal.html}{http://www.risc.jku.at/people/jschicho/pub/gonal.html} for the full computations.

\begin{example}
The canonical model of the plane curve $x_0^4 x_2^5+x_1^9+x_1^2 x_2^7=0$ over $\qp$ has genus 10 and Betti table
\[ \begin{array}{ccccccccc}
  1 &  - &  - &  - &  - &  - &  - &  - & - \\
  - & 28 & 105& 168& 154& 70 &  6 &  - & - \\
  - &  - &  6 & 70 & 154& 168& 105& 28 & - \\
  - &  - &  - &  - &  - &  - &  - &  - & 1. \end{array}  \]
The curve is 4-goneric. The unique 3-dimensional scroll $X$ has Betti table
\[ \begin{array}{ccccccc}
  1 &  - &  - &  - &  - & -  & - \\
  - & 21 & 70 & 105& 84 & 35 &  6. \end{array} \]
Its Lie algebra has dimension~14, with a Levi subalgebra isomorphic to $\mathfrak{sl}_2$. The induced $\mathfrak{sl}_2$-module of linear forms splits into three irreducible submodules of dimension~2,3,5. The projection defined by the forms defining the two-dimensional submodule define the unique projection to $\PP^1$. It corresponds to the projection of the plane model from the 5-fold point $(1{:}0{:}0)$.
\end{example}

\begin{example}
A similar example starts with the plane curve $x_0^9+x_2^9+x_1^4 x_2^5=0$. Its canonical model has genus 12 and Betti table
\[ \begin{array}{ccccccccccc}
  1 &  - &  - &  - &  - &  - &  - &  - &  - &  - & - \\
  - & 45 & 231& 558& 840& 798& 468& 147&  8 &  - & - \\
  - &  - &  8 & 147& 468& 798& 840& 558& 231& 45 & - \\
  - &  - &  - &  - &  - &  - &  - &  - &  - &  - & 1. \end{array}  \]
It is also 4-goneric, and the unique 3-dimensional scroll $X$ has Betti table
\[ \begin{array}{ccccccccc}
  1 &  - &  - &  - &  - & -  & -  &  - & - \\
  - & 36 & 168& 378& 504& 420& 216& 63 & 8. \end{array} \]
Its Lie algebra has dimension~16, with a Levi subalgebra isomorphic to $\mathfrak{sl}_2$. The three irreducible submodules have dimension 2,4,6, and the 2-dimensional submodule gives the unique 4:1 map.
\end{example}

\subsection{Remarks about gonal maps of plane curves}
In the two previous examples, the unique gonal map is induced by projection from a singularity of a plane model. More generally, suppose that you have a plane model $C\subset \pp^2$ of degree $d$ of your input curve. Projection from a singularity of highest multiplicity $\nu$ leads to a map $C\to \pp^1$ of degree $d-\nu$, so that
$$
\gon(C)\le d-\nu.
$$
It turns out that in many cases, this projection is a gonal map and equality holds. Let 
$$
\delta=\frac{(d-1)(d-2)}{2}-g
$$
be the usual $\delta$-invariant and define the quadratic function:
$$
Q(x) = x(x - d) + d + \delta  - \nu.
$$
We have the following sufficient criterion:

\begin{proposition}\label{p_plane}\cite[Proposition 2]{OS}
Suppose $d/\nu \ge 2$. If $Q([d/\nu])\le 0$, then $\gon(C)= d - \nu$. 
\end{proposition}

If for instance we can find a nodal plane model of the input curve with 
$$
{\rm Number\,\,of\,\,nodes} \,\le \, (d/2-1)^2+1
$$
then $\gon(C)=d-2$. Note that \cite[Proposition 2]{OS} gives other finner sufficient conditions that involve the all multiplicity sequence of $C$. Moreover, there exists too lower bounds for the gonality of plane curves. For instance, if $d/\nu \ge 2$ and $Q([d/\nu])> 0$ then $\gon(C)\ge d-\nu -Q([d/\nu])$ \cite[Theorem $3$]{Sak}. 

\begin{remark}
Note that the curve may have more gonal maps than those induced by the projection from the singularities. For instance, a plane curve of degree $6$ with $4$ nodes have $5$ $g^1_4$'s, the fifth one being the pencil of conics passing throw the nodes.
\end{remark}

Hence, as soon as getting a plane model of the curve is not too expensive, we may try to check whether the previous sufficient condition holds. In that case, computing a gonal map is almost trivial and doesn't require to compute the syzygies of a canonical model.

\section{Radical parametrization of tetragonal curves}\label{S5}

A morphism $C\to \pp^1$ of degree $\le 4$ being invertible by radicals, Theorem \ref{t3} follows directly from Theorem \ref{t1}. Nevertheless, Algorithm \ref{alg3} is not efficient for curves of high genus and the faster Algorithm \ref{alg4} only applies to goneric curves. We develop here a faster algorithm for computing gonal maps of curves of gonality $\le 4$.

Since there are efficient factorization and normalization algorithms for curves, we can assume that the input curve $C$ is smooth and absolutely irreducible (up to  extend the ground field if necessary). By Theorem \ref{GreenLazarsfeld} and Theorem \ref{GonalityClifford}, curves with linear colength $\ell(C)>2$ have gonality strictly greater than $4$. The case of curves with $\ell(C)\le 1$ being well-known (\cite{SS} or Algorithm \ref{alg3} for small genus), we are left to pay attention to smooth irreducible canonical curves with $\ell(C)=2$.

\subsection{Constructive classification of canonical curves with $\ell(C)=2$.} 

We use notations of Subsubsection \ref{4.3.1}. Each term of the resolution $(F_{\bullet})$ of $S_C$ decomposes as 
$$
F_i=L_i\oplus Q_i=S(-i-1)^{\beta_{i,i+1}}\oplus S(-i-2)^{\beta_{i,i+2}}
$$ 
and we pay here attention to the linear syzygies $B_i: Q_{i}\rightarrow Q_{i-1}$.  The following result is a constructive version of a theorem of Schreyer \cite{S2}:

\begin{theorem}\label{pIX}
Let $C$ be a smooth irreducible canonical curve with $\ell(C)=2$. Then $\beta_{1,3}= 0$ and two cases can occur (three if $\chara(K)=2$):
\begin{enumerate}
\item \textbf{$\beta_{2,4}=g-4$.} Then $C$ is $4$-goneric, contained in a unique $3$-dimensional scroll $X$ with homogeneous ideal $I_X=\Ann(\coker(B_3))$.
\item \textbf{$\beta_{2,4}=\binom{g-2}{2}-1$.} Then $C$ is contained in a unique arithmetically Cohen-Macaulay surface $Z$ of degree $g-1$. We have 
$$
\ker\Big(B_3^t: S(4)^{\beta_{2,4}} \rightarrow S(5)^{\beta_{3,5}}\Big)\cong S(2)
$$
and the entries of the syzygies matrix $M:S(2)\longrightarrow S(4)^{\beta_{2,4}}$ generate the homogeneous ideal $I_Z$. Two sub-cases can occur:

\begin{enumerate}
\item The ideal generated by the entries of the first syzygy module of $I_Z$ is the maximal ideal. Then $5\le g\le 10$ and $Z$ is a possible singular Del Pezzo surface whose conical fibrations induce the $g^ 1_4$'s of $C$. In particular, $C$ has no $g^1_4$ iff $g=10$, that is, if $Z\simeq \pp^2$ and $C$ is isomorphic to a smooth plane sextic.

\item The ideal generated by  the entries of the first syzygies module of $I_Z$ defines a point in $v\in \pp^{g-1}$. Then $Z$ is a cone over an elliptic normal curve $E\subset \pp^{g-1}$ with vertex $v$. It induces a diagramm
\begin{equation}\label{eval}
\begin{array}{ccccc}
  C  &\subset  &  & Z &\\
 2:1 \searrow &  &  \swarrow & &  \\
 &  E &  & &\stackrel{2:1} \longrightarrow \,\, \pp^1 
\end{array}
\end{equation}
and $C$ has infinitely many $g^1_4$'s induced by the $g^1_2$'s of $E$.
\end{enumerate}
\item \text{$\beta_{2,4}=1$. } Then  $\chara(K) =2$, genus $g=7$ and  $C$ is not 4-gonal.
\end{enumerate}
\end{theorem}

\begin{proof}
\textit{Case 1).} Follows from Theorem \ref{pgoneric}.

\textit{Case 2).} By \cite{S2}, we know that $C$ is a conical section of an arithmetically Cohen Macaulay rational surface $Z\subset \pp^{g-1}$ of minimal degree $g-1$. 
Hence, the isomorphism $\mathcal{O}_C\cong \mathcal{O}_Z/\mathcal{O}_Z(-2)$ leads to an exact sequence of homogeneous coordinate rings
$$
0\rightarrow S_Z(-2)\rightarrow S_Z \rightarrow  S_C\rightarrow 0.
$$
The resolution $(F_{\bullet})$ of $S_C$ thus coincides with the mapping cone \cite{Eis}
\begin{equation}\label{mappingCone}
F_{\bullet}\cong \mathcal{C}\big(G_{\bullet}(-2)\rightarrow G_{\bullet}\big):=G_{\bullet}\oplus G_{\bullet-1}(-2),
\end{equation}
where $(G_{\bullet})$ is the minimal free resolution of $S_Z$ (of length $g-3$). The curve $C\in |\mathcal{O}_Z(2H)|$ being a very ample divisor of $Z$, the isomorphisms 
$$
\omega_Z(2H)\otimes\mathcal{O}_C\cong\omega_C\cong \mathcal{O}_{C}(H)
$$
forces  $\omega_Z=\mathcal{O}_Z(-H)$ so that the surface $Z\subset \pp^{g-1}$ is anti-canonically embedded. As for canonical curves, it follows by adjunction that $(G_{\bullet})$ is self-dual (up to a twist). The syzygy matrices $G_i\to G_{i-1}$ have linear entries, except the first and last matrices with entries the quadratic forms that generate $I_Z$. For degree reasons, it follows from (\ref{mappingCone}) that the Betti tables of $C$ and $Z$ have the following respective particular shapes:
\begin{equation*}
\begin{array}{ccccccccc}
1  &  -  &  - &  \quad - &     &    - \quad  & -  & - & -\\ 
- &  b_1+1  &  b_2  &  \quad b_3  & \quad \cdots \quad &  b_2 \quad & b_1 & -  &  -\\ 
- &  -  &  b_1  &  \quad b_2  & \quad \cdots \quad &   b_3 \quad &    b_2  &  b_1+1  &  -\\
- & -   & - & \quad - &      &    - \quad  & -  &  - & 1 \\
&&&&  \cup &&&& \\
1 &  -  &  - & - &   &    -    & -  & - &\\ 
- &  b_1  &  b_2  &  b_3  & \quad \cdots \quad &  b_2  &  b_{1}  &- &   \\ 
- &  -  & - &  -  &     &  -  & -  &  1 &
\end{array}
\end{equation*}
Hence, $\beta_{2,4}=b_1$ coincides with the number $\binom{g-2}{2}-1$ of generators of $I_Z$. Moreover, we get by symmetry that
$$
\ker\Big(B_3^t: S(4)^{\beta_{2,4}} \rightarrow S(5)^{\beta_{3,5}}\Big)\cong S(2)
$$
and that the entries of the syzygies matrix $S(2)\longrightarrow S(4)^{\beta_{2,4}}$ generate the homogeneous ideal $I_Z$. Note that $Z$ is uniquely determined by $C$ by construction.

The one-to-one correspondance between the $g^1_4$'s on $C$ and the conical fibrations of $Z$ is clear: A conical fibration of $Z$ is a map $p:Z\to \pp^1$ whose any fiber $Q$ is a degree $2$ curve, that is, $\deg(\mathcal{O}_Z(H)_{|Q})= 2$. Since $C\in |\mathcal{O}_Z(2H)|$, we deduce that
$C\cdot Q= 4$ and $p$ induces a $g^1_4$ on $C$. Conversely, a degree $4$ map $q:C\to \pp^1$
leads to a scroll $X$ such that $C\subset Z\subset X$ and whose fibration $r:X\to \pp^1$ restricts to $q$ on $C$. A fiber curve $Q\subset Z$ of $r_{|Z}$ intersecting $C$ in $4$ points, it has degree $2$ and $r_{|Z}$ is a conical fibration. 

By classification theory \cite{Na}, an arithmetically Cohen Macaulay rational surface $Z\subset \pp^{g-1}$ of almost minimal degree $g-1$ containing a smooth curve is either a possible singular Del Pezzo surface or a cone over an elliptic curve with vertex a point. It is a cone if and only if the entries of the matrix of its first syzygy module only depend on $g-1$ variables after some good choice of coordinates, these entries defining the vertex of the cone. Otherwise, the entries of the first syzygy module generate the maximal ideal. 

\textit{Del Pezzo case.} An anti-canonical Del Pezzo surface has degree $\le 9$ and a curve of genus $g< 5$ has gonality $<4$, hence the bounds $5\le g\le 10$. There are a finite numbers of conical families, hence a finite number of $g^1_4$'s. There are no $g^1_4$ if and only if $Z\simeq \pp^2$ ($g=10$), that is, if $C$ is isomorphic to the smooth plane sextic.

\textit{Elliptic cone case.} In the elliptic cone case, we have here a $\pp^1$-fibration
$Z\to E$ over a plane elliptic curve $E$. Each $g^1_2$ on $E$ (induced by the projection from a point on $E$) leads to a conical fibration $Z\to E \stackrel{2:1}\longrightarrow \pp^1$ of $Z$,
each conic being a a union of two intersecting lines (a fiber is connected). It follows that $Z$ has a rational curve $\wt{E}$ of double points. Any conical fibration of $Z$ arises in this way: the fibers have to meet the curve $\wt{E}$, so there are singular, {\it i.e} union of two lines. We saw that any conical fibration $Z\to \pp^1$ leads to a $g^1_4$ on $C$, so the curve $C$ is necessarily a double cover of $E$ (each line of the fibration intersects $C$ in two points). We have 
\begin{equation*}
\begin{array}{ccccc}
  C  &\subset  &  & Z &\\
 2:1 \searrow &  &  \swarrow & &  \\
 &  E &  & &\stackrel{2:1} \longrightarrow \,\, \pp^1 
\end{array}
\end{equation*}
and the $g^1_4$'s of $C$ are one-to-one with the (infinitely many) $g^1_2$'s of $E$. 

\textit{Case 3)}: see \cite{S1}.
\end{proof}

\begin{remark} If there is a unique $g^1_4$ (the scroll case) it might happen too that it factorizes through an hyperelliptic curve (of genus $g'\ge 2$) which, in contrast to the elliptic case, has a unique $g^1_2$. For $g > g'$ high enough, we can extract the factorizability conditions directly from the graded Betti numbers of $C$ (see \cite{S1}).
\end{remark}

\begin{example} Starting with a planar sextic with three nodes that are not collinear, we obtain a canonical curve $C$ of genus 7 in $\PP^6$, with Betti table
\[ \begin{array}{cccccc}
  1 &  - &  - &  - &  - & - \\
  - & 10 & 16 &  9 &  - & - \\
  - &  - &  9 & 16 & 10 & - \\
  - &  - &  - &  - &  - & 1. \end{array}  \]
By Theorem~\ref{pIX}, $C$ is contained in a unique surface $Z$ of degree~6, which in this case is a Del Pezzo surface, with Betti table
\[ \begin{array}{ccccc}
  1 &  - &  - &  - &  - \\
  - &  9 & 16 &  9 &  - \\
  - &  - &  - &  - &  1. \end{array}  \]
It has three conic fibrations, which restrict to the three 4:1 maps to $\PP^1$ that correspond to the projections of the sextic from the three nodes.
\end{example}

\begin{example} 
A similar example starts with a plane model with equation $x_0^8+x_2^8+x_1^4 x_2^4=0$ over $\qp$. Its canonical embedding in $\PP^8$ has Betti table
\[ \begin{array}{cccccccc}
  1 &  - &  - &  - &  - &  - &  - & - \\
  - & 21 & 64 & 90 & 64 & 20 &  - & - \\
  - &  - & 20 & 64 & 90 & 64 & 21 & - \\
  - &  - &  - &  - &  - &  - &  - & 1. \end{array}  \]
The unique surface $Z$ is a weak Del Pezzo surface of degree~8 with a node, with Betti table
\[ \begin{array}{ccccccc}
  1 &  - &  - &  - &  - &  - &  - \\
  - & 20 & 64 & 90 & 64 & 20 &  - \\
  - &  - &  - &  - &  - &  - &  1 \end{array}  \]
and a single conic fibration. It restricts to the unique gonal $g^1_4$ (that has to be counted with multiplicity $2$ since the fibers of $Z$ are double lines), corresponding to the projection of the plane octic from the 4-fold point $(0{:}1{:}0)$.
\end{example}

\subsection{Radical parametrization of tetragonal curves (proof of Theorem \ref{t3})}

We obtain the following algorithm:

\begin{algo}[Radical parametrization of curves with $\gon(C)\le 4$] \noindent

\textbf{Input:} An irreducible and reduced projective curve $C$.

\textbf{Output:} A radical parametrization if $\gon(C)\le 4$ or ``$\gon(C)> 4$''.

\textit{Step 1.} Compute the homogeneous ideal $I_C$ of the canonical embedding of the normalization of $C$.

\textit{Step 2.} Compute the first and second syzygy modules of $S_C$. If $\ell(C)\le 1$ compute the gonal map by well known algorithms and go to step 5. If $\ell(C)\ge 3$, return "gonality greater than $4$". 

\textit{Step 3.} If $\beta_{2,4}=g-4$, compute the unique $g^1_4$ with algorithm \ref{alg4} and go to step $5$. 

\textit{Step 4.} If $\beta_{2,4}> g-4$, compute $I_Z$ and detect if $Z$ is a cone with Theorem \ref{pIX}. There are three possible cases:

\hspace{4mm}\textit{4.a.} If $Z$ is not a cone and $g=10$, return "smooth plane sextic". 

\hspace{4mm}\textit{4.b.} If $Z$ is not a cone and $g<10$, compute a conical fibration $Z\to \pp^1$ thanks to the Lie algebra method \cite{schi}. Restrict it to a $g^1_4$ on $C$ and go to step 5.

\hspace{4mm}\textit{4.c.} Else, $Z$ is a cone over a genus $1$ curve $E\subset \pp^{g-2}$. Deduce $I_E$ from $I_Z$ and deduce the $2:1$ map $C\to E$. Pick a point $p$ on $E$ (you might use a field extension of degree $\le \deg(E)=g-1$ by elimination theory). Compute the induced $2:1$ map  $|2p|: E\to \pp^1$. Deduce a $g^1_4$ of $C$.

\textit{Step 5.} Invert the gonal map by radicals. This is well-known, see \cite{Sweedler}.
\end{algo}

This algorithm is correct thanks to Theorem \ref{pIX}. Theorem \ref{t3} follows. $\qquad{\square}$

\begin{remark}[Radical parametrization of curves of higher gonality]\label{radical}

We might hope to introduce geometric considerations in the picture to detect resolubility of some particular curves of gonality $>4$, namely those curves for which the gonal map factors threw intermediary curve. For instance, if the gonal map of an $8$-gonal curve $C$ factors threw a $4$-gonal curve $B$,
\begin{equation*}\label{eval}
 C \stackrel{2:1}\longrightarrow B \stackrel{4:1}\longrightarrow \pp^1
\end{equation*}
then $C$ is parametrizable by radicals (Note: there exist finite covers $C\to B$ for which the gonal map of $C$ does not factor throw $B$ \cite{BKP}). Let us mention here \cite[Theorem 1.1]{Coskun}: if $C$ is a smooth canonical curve of genus $g> 3k+12$ such  that the smallest degree surface containing $C$ has degree $g+k$, then $C$ is a double cover of a curve of genus $(k+3)/2$ (the case $(g,k)=(10,-1)$ corresponds to the case $2)b)$ of Theorem \ref{pIX}). In particular, if $k\le 9$ or $\gon(C)\le 8$, then $C$ is parametrizable by radicals. 

It would be interesting to try to guess and compute the complete factorization of the gonal map threw intermediary curves by using syzygies, as shown in \cite{S2} for tetragonal curves. In the same spirit, there are closed relations between small degree surfaces containing the curve and the construction of a gonal map \cite{CK} and we may wonder if we can extract the ideals of such surfaces from the resolution of $I_C$, generalizing Theorem \ref{pIX}. 
\end{remark}

%
%
%
%
%



\end{document}